\documentclass[a4paper]{amsart}

\usepackage{amsfonts}
\usepackage{verbatim}
\usepackage{amssymb}
\usepackage{amscd}
\usepackage{amsmath}
\usepackage{latexsym}
\usepackage{framed, color}

\usepackage[dvips]{graphicx}
\usepackage{psfrag}

\newtheorem{thm}{Theorem}[section]
\newtheorem{cor}[thm]{Corollary}
\newtheorem{prop}[thm]{Proposition}
\newtheorem{lem}[thm]{Lemma}

\theoremstyle{definition}
\newtheorem{dfn}[thm]{Definition}

\theoremstyle{remark}
\newtheorem{rem}[thm]{Remark}
\newtheorem{problem}[thm]{Problem}
\newtheorem{ex}[thm]{Example}

\makeatletter
 
 \@addtoreset{equation}{section}
\makeatother

\newcommand{\R}{\mathbb{R}}

\newcommand{\Int}{\operatorname{Int}}

\renewcommand{\tilde}{\widetilde}

\def\spmapright#1{\smash{%
 \mathop{\hbox to 1.3cm{\rightarrowfill}}
  \limits^{#1}}}
\def\spmapleft#1{\smash{%
 \mathop{\hbox to 1.3cm{\leftarrowfill}}
  \limits^{#1}}}

\title[Reeb spaces of smooth functions on manifolds]
{Reeb spaces of smooth functions on manifolds}
\dedicatory{Dedicated to Professor Toshizumi Fukui
on the occasion of his 60th birthday}
\author{Osamu Saeki} 
\address{Institute of Mathematics for Industry,
Kyushu University,
Motooka 744, Nishi-ku, Fukuoka 819-0395, Japan}

\email{saeki@imi.kyushu-u.ac.jp}
\date{\today}


\begin{document}
\begin{abstract}
The Reeb space of a continuous function
is the space of connected components 
of the level sets.
In this paper we first prove that the Reeb space
of a smooth function on a closed
manifold with finitely many critical values
has the structure of a finite graph without
loops.
We also show that an arbitrary finite graph 
without loops can be realized as the Reeb space
of a certain smooth function on a closed manifold
with finitely many critical values, where
the corresponding level sets can also be preassigned.
Finally, we show that a continuous map of a smooth closed
connected manifold to a finite connected
graph without loops that induces
an epimorphism between the fundamental groups is
identified with the natural quotient map to the Reeb space
of a certain smooth function with finitely many
critical values, up to homotopy.
\end{abstract}

\maketitle

\section{Introduction}\label{section1}

In this article, we prove three theorems on
Reeb spaces of smooth functions on compact manifolds
of dimension $m \geq 2$
that have finitely many \emph{critical values}.
Here, the \emph{Reeb space} is the space of
connected components of the level sets, endowed with
the quotient topology (for a precise definition,
see \S\ref{section2}).

The first theorem (Theorem~\ref{thm:Reeb})
states that the Reeb space of such a function always
has the structure of a finite (multi-)graph without loops.
Many of the results in the literature concern Morse
functions or smooth functions with finitely
many \emph{critical points} \cite{Izar1, Izar2, Izar3, Izar4, Izar5,
R, Sh}, and our theorem
generalizes them. Note that the graph structure of the
Reeb space of such a smooth function with finitely
many critical values satisfies that the vertices correspond exactly to
the connected components of level sets that contain critical points.
In the literature, a Reeb space is often called a 
\emph{Reeb graph}, and
our theorem justifies the terminology.
We note that the same result also holds for every smooth function
on an arbitrary compact manifold of dimension $m \geq 2$ possibly with boundary
provided that the function and its restriction
to the boundary both have finitely many critical values.

The second result of this paper is a realization
theorem (Theorem~\ref{thm2}). We show that, for each $m \geq 2$, 
an arbitrary finite (multi)-graph without loops
can be realized as the Reeb space of a smooth function on a 
closed $m$--dimensional manifold with finitely many critical values.
Our result is even stronger: we can preassign the diffeomorphism types of 
the components
of level sets for points in the graph. More precisely, to each edge we assign a
closed connected $(m-1)$--dimensional manifold and to each vertex
a compact connected $m$--dimensional manifold so that certain 
consistency conditions are satisfied:
then, a graph with such a pre-assignment, called
an $m$--decorated graph, can always be realized as the Reeb space
of a smooth function with finitely many critical values on a closed
$m$--dimensional manifold in such a way that
a point on an edge corresponds to the preassigned $(m-1)$--dimensional manifold
and a vertex corresponds to the preassigned $m$--dimensional manifold.

The third result of this paper also concerns the realization of a graph
as the Reeb space (Theorem~\ref{thm3}). 
In our second result, the source manifold
on which the function is constructed is not preassigned. On the other
hand, in our third theorem, we fix the source closed
connected manifold $M$ of dimension
$m \geq 2$ and
first consider a continuous map of $M$ into a connected (multi-)graph 
without loops that induces an epimorphism between the fundamental groups.
Then, we show that such a map is homotopic to the quotient map
to the Reeb space
of a smooth function on $M$ with finitely many critical values,
where the Reeb space is identified with the given graph.
Note that the condition on the fundamental group is known to be necessary.
A similar result has been obtained by
Michalak \cite{M1, M2} (see also Gelbukh \cite{G3},
Marzantowicz and Michalak \cite{MM}):
for $m \geq 3$, one can realize
a given graph as the Reeb space of a Morse function on a closed
$m$--dimensional manifold
up to homeomorphism.
Our theorem is slightly different from such results
in that we not only realize the topological
structure of a given graph but we also 
realize the given graph structure.
We construct smooth functions with finitely
many critical values such that the images by the quotient map
of the level set components containing critical points 
exactly coincide with the vertices of the given graph.

The paper is organized as follows. In \S\ref{section2}, we review the
definition and certain properties of Reeb spaces, and present
the problems addressed in this paper. In \S\ref{section3}, we prove that
the Reeb space of a smooth function on a closed manifold with
finitely many critical values is a graph. The key to the
proof is Lemma~\ref{lem2}, which guarantees that the
number of connected components of a level set
corresponding to an isolated critical value is always finite.
We use some general topology techniques to prove this lemma.
In \S\ref{section3.5}, we introduce the notion of a path
Reeb space, which is the space of all path-components of
level sets. We show that for the path Reeb space,
our first theorem does not hold in general: in fact,
even for a smooth function on a closed manifold with finitely
many critical values, the path Reeb space may not be a
$T_1$--space, and hence may not have the structure of a graph.
However, we also show that the path Reeb space and
the usual Reeb space coincide with each other for an arbitrary
smooth function on a closed manifold with finitely many critical points. 
In \S\ref{section4}, we prove the second theorem. In order to 
construct a desired function on a manifold, we first
construct non-singular functions corresponding to edges
and constant functions corresponding to the vertices.
Then, we glue them together ``smoothly'' using our consistency condition.
In \S\ref{section5}, we prove the third theorem. By using the assumption
on the fundamental groups, we first find $(m-1)$--dimensional
submanifolds in the source manifold $M$ that correspond
to the edges using surgery techniques. Then, we use
the techniques employed in the proof of the second theorem.

Some of the results in this paper have been announced in \cite{Sa20}.

Throughout the paper, 
all manifolds and maps between them are smooth
of class $C^\infty$ unless otherwise specified. 
The symbol ``$\cong$'' denotes a diffeomorphism between
smooth manifolds.

\section{Reeb space}\label{section2}


Let $f : X \to Y$ be a continuous map
between topological spaces.
For two points $x_0, x_1 \in X$, we write
$x_0 \sim x_1$ if $f(x_0) = f(x_1)$ and $x_0, x_1$
lie on the same connected component
of $f^{-1}(f(x_0)) = f^{-1}(f(x_1))$.
Let $W_f = X/\!\sim$ be the quotient space 
with respect to this equivalence relation: i.e.\ $W_f$
is a topological space endowed with the quotient topology.
Let $q_f : X \to W_f$ denote the quotient map.
Then, there exists a unique map $\bar{f} : W_f \to Y$ that 
is continuous and makes the 
following diagram commutative:
\begin{eqnarray*}
X \!\!\!\! & \spmapright{f} & \!\!\!\! Y \\
& {}_{q_f}\!\!\searrow \quad \qquad \nearrow_{\bar{f}} & \\
& \,W_f. &
\end{eqnarray*}
The space $W_f$ is called the
\emph{Reeb space} of $f$, and
the map $\bar{f} : W_f \to Y$ is called
the \emph{Reeb map} of $f$. The decomposition
of $f$ as $\bar{f} \circ q_f$ as in the above
commutative diagram is called the \emph{Stein
factorization} of $f$ \cite{L1}. For a schematic example,
see Figure~\ref{fig1}.

\begin{figure}[h]
\centering
\psfrag{f}{$f$}
\psfrag{M}{$X$}
\psfrag{N}{$Y$}
\psfrag{W}{$W_f$}
\psfrag{bf}{$\bar{f}$}
\psfrag{q}{$q_f$}
\includegraphics[width=0.9\linewidth,height=0.4\textheight,
keepaspectratio]{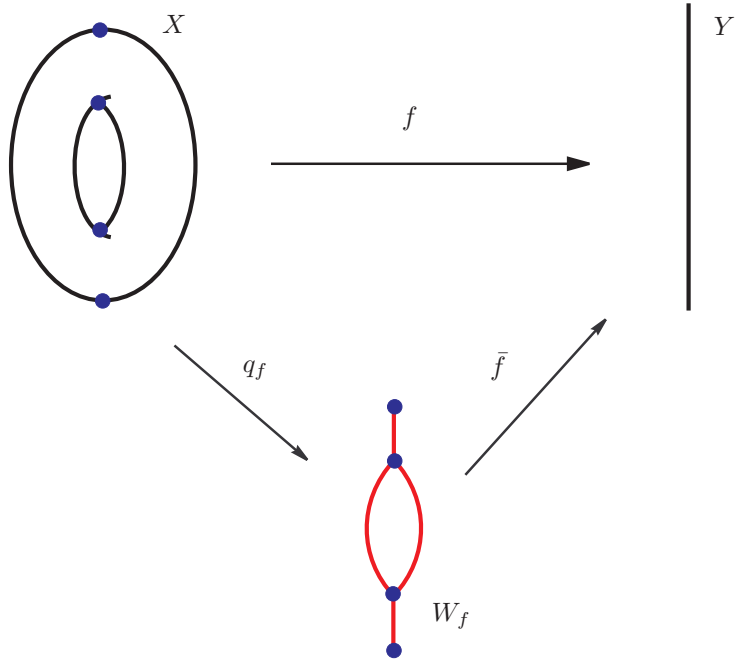}
\caption{Example of a Reeb space and a Stein factorization}
\label{fig1}
\end{figure}

In this article, a smooth real-valued
function on a manifold is called a \emph{Morse
function} if its critical points are
all non-degenerate. Such a Morse function
is \emph{simple} if its restriction to
the set of critical points is injective.

Furthermore, in the following, a 
\emph{graph} means a finite ``multi-graph''
which may contain multi-edges and/or loops.
When considered as a topological space, it
is a compact $1$--dimensional
CW complex. A manifold is \emph{closed}
if it is compact and has no boundary.

It is known that the Reeb space of a
Morse function
on a smooth closed manifold
has the structure of a graph, 
which is often called a 
\emph{Reeb graph} \cite{R} or is sometimes
called a \emph{Kronrod--Reeb graph} \cite{Sh}.
This fact has been first stated in \cite{R} without proof.
A proof for simple Morse functions can be found
in \cite[Teorema~2.1]{Izar3} (see also \cite{Izar1}--\cite{Izar5}).

If we use several known results, this fact for
Morse functions
can also be proved as follows, for example.
First, Morse functions on closed manifolds are triangulable
(for example, by a result of Shiota \cite{Sh}).
Then, by \cite{HS}, its Stein factorization is triangulable
and consequently the quotient space, i.e.\ the
Reeb space, is a $1$--dimensional polyhedron.
Therefore, it has the structure of a graph.

Furthermore, the following properties are known
for Morse functions.
\begin{enumerate}
\item The vertices of a Reeb graph
correspond to the components of level sets that
contain critical points.
\item The restriction of the Reeb map on each
edge is an embedding into $\R$. In particular, no edge is
a loop.
\end{enumerate}

Then, the following natural problems arise.

\begin{problem}
Given a smooth function on a closed
manifold,
does the Reeb space
always have the structure of a graph?
\end{problem}

\begin{problem}
Given a graph, is it realized
as the Reeb space of a certain smooth function?
\end{problem}

In this article, we consider these problems
and give some answers.

\section{Reeb graph theorem}\label{section3}

Let $M$ be a smooth closed manifold of
dimension $m \geq 2$ and $f : M \to \R$
a smooth function.
Then, we have the following.

\begin{thm}\label{thm:Reeb}
If $f$ has at most finitely many critical values, then
the Reeb space 
$W_f$ has the structure
of a graph.
Furthermore, $\bar{f} : W_f \to \R$ is an embedding
on each edge.
\end{thm}

\begin{rem}
The above theorem has been known for smooth functions with finitely
many critical \emph{points}. A proof can be found in \cite{Sh}.
\end{rem}

\begin{rem}
(1) If $f$ has infinitely many critical
values, then the above theorem does not hold in general.
Some examples will be presented later.

(2) The same result holds also
for smooth functions $f : M \to \R$
on compact manifolds of dimension $m \geq 2$ with
$\partial M \neq \emptyset$,
provided that both $f$ and $f|_{\partial M}$
have at most finitely many critical values.
This can be proved by the same argument as in the proof
of Theorem~\ref{thm:Reeb} given below.
\end{rem}

\begin{rem}
It is known that if $f : M \to N$ with
$m = \dim{M} > \dim{N} = n \geq 1$
is a smooth map between manifolds with
$M$ being compact, and if $f$
is triangulable, 
then its Reeb space $W_f$
has the structure of an 
$n$--dimensional
finite simplicial complex (or a compact polyhedron)
in such a way that $\bar{f}$ is an embedding
on each simplex \cite{HS}.
In particular, if $f$ is $C^0$--stable,
then the Reeb space $W_f$ is an $n$--dimensional polyhedron.
Therefore, if a smooth function
$f : M \to \R$ on a compact
manifold is triangulable (for example,
if $f$ is a Morse function),
then it follows that $W_f$ is a graph.

As we will see later, 
there exist smooth functions with
finitely many critical values that are
not triangulable.
\end{rem}

For the proof of Theorem~\ref{thm:Reeb}, we need the
following. In the following, for a subset
$S$ of a manifold, $\overline{S}$ denotes the
closure in the ambient manifold.

\begin{lem}\label{lem1}
Let $f : M \to \R$ be a smooth function on a closed
manifold. Suppose $c \in \R$ is a critical value 
such that $[c, c+2\varepsilon)$ does not contain
any critical value other than $c$ for some $\varepsilon > 0$.
Then, the number of connected components of
$$\overline{f^{-1}((c, c+\varepsilon])} \setminus 
f^{-1}((c, c+\varepsilon])$$
does not exceed that of $f^{-1}(c+\varepsilon)$, which
has finitely many connected components.
\end{lem}

\begin{proof}
Let us first consider the case where $f^{-1}(c+\varepsilon)$
is non-empty and connected. For each $n \geq 1$, as 
$$f^{-1}\left(\left(c, c+ \frac{\varepsilon}{n}\right]\right)
\cong f^{-1}(c+\varepsilon) \times \left(c, c+ \frac{\varepsilon}{n}\right]$$
is connected, its closure, 
$$X_n = \overline{f^{-1}\left(\left(c, 
c+ \frac{\varepsilon}{n}\right]\right)},$$ 
is also connected. Note that
$X_1 \supset X_2 \supset X_3 \supset \cdots$.
We see easily that
$$\bigcap_{n=1}^\infty X_n =
\overline{f^{-1}((c, c + \varepsilon])}
\setminus f^{-1}((c, c+\varepsilon]).$$
Note that this is compact.

Suppose that $\cap_{n=1}^\infty X_n$ 
is not connected. Then, it is decomposed
into the disjoint union $A \cup B$ of non-empty closed sets
$A$ and $B$. As $A$ and $B$ are compact and $M$ is Hausdorff,
there exist disjoint open sets $U$ and $V$ of $M$ such that
$A \subset U$ and $B \subset V$.
Note that then we have $\cap_{n=1}^\infty X_n \subset U \cup V$.

Suppose for all $n$, we have $X_n \not\subset U \cup V$.
Set $\tilde{X} = X_1 \setminus (U \cup V)$, which
is non-empty and compact. As $F_n = X_n \setminus (U \cup V)$, $n \geq 1$,
is a family of closed sets of $\tilde{X}$ which has finite
intersection property, we have
$\cap_{n=1}^\infty F_n \neq \emptyset$,
which contradicts the fact that 
$\cap_{n=1}^\infty X_n \subset U \cup V$.
This shows that we have $X_{n_0} \subset U \cup V$
for some $n_0 \geq 1$.

Then, we have $X_{n_0} \cap U \supset A \neq \emptyset$ and
$X_{n_0} \cap V \supset B \neq \emptyset$, which
contradicts the connectedness of $X_{n_0}$.
Consequently $\cap_{n=1}^\infty X_n$ is connected.

If $f^{-1}(c+\varepsilon)$ is empty, then the
consequence of the lemma trivially holds.
Let us now consider the general case
where $f^{-1}(c+\varepsilon) \neq \emptyset$ may not be connected.
In this case, $f^{-1}((c, c+\varepsilon])$ has finitely
many connected components, say $K_1, K_2, \ldots, K_k$.
By the same argument as above, we can show that
$\overline{K_i} \setminus K_i$ is connected for each $1 \leq i \leq k$.
Therefore, the number of connected components of
$$\overline{f^{-1}((c, c+\varepsilon])} \setminus 
f^{-1}((c, c+\varepsilon]) = \bigcup_{i=1}^k (\overline{K_i}
\setminus K_i)$$
is at most $k$. This completes the proof
of Lemma~\ref{lem1}.
\end{proof}

\begin{rem}\label{rem:lem1}
Note that in Lemma~\ref{lem1}, we can similarly show
that if $(c-2\varepsilon, c]$ does not have
any critical value other than $c$ for some $\varepsilon > 0$, then
the number of connected components of
$$\overline{f^{-1}([c-\varepsilon, c))} \setminus 
f^{-1}([c-\varepsilon, c))$$
does not exceed that of $f^{-1}(c-\varepsilon)$, which
has finitely many connected components.
\end{rem}

\begin{rem}
In Lemma~\ref{rem:lem1}, the number of connected
components of $$\overline{f^{-1}((c, c+\varepsilon])} \setminus 
f^{-1}((c, c+\varepsilon])$$
can be strictly smaller than that of $f^{-1}(c+\varepsilon)$.
For an example, see Figure~\ref{fig12}.
\end{rem}

\begin{figure}[h]
\centering
\psfrag{f}{$f$}
\psfrag{R}{$\R$}
\psfrag{M}{$M$}
\psfrag{c}{$c$}
\psfrag{cc}{$c + \varepsilon$}
\includegraphics[width=0.9\linewidth,height=0.4\textheight,
keepaspectratio]{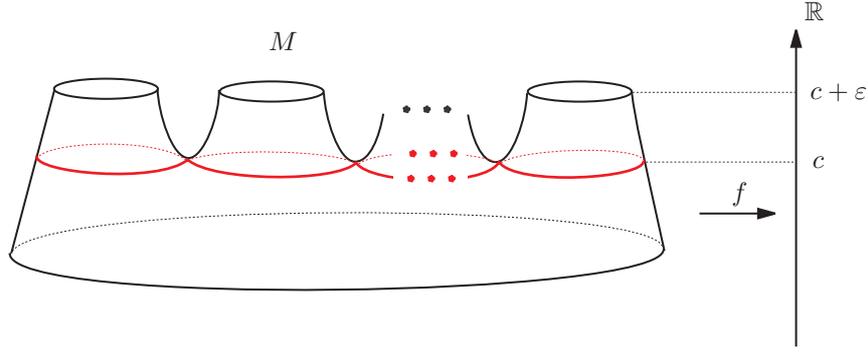}
\caption{Number of components of $f^{-1}(c)$ can
be strictly smaller than that of $f^{-1}(c + \varepsilon)$.}
\label{fig12}
\end{figure}

\begin{lem}\label{lem2}
Let $f : M \to \R$ be a smooth function on a closed
manifold. Suppose $c \in \R$ is an isolated critical value 
of $f$. Then, $f^{-1}(c)$ has at most finitely
many connected components.
\end{lem}

\begin{proof}
Let $A_\lambda$, $\lambda \in \Lambda$, be
the connected components of $f^{-1}(c)$.
Note that $A_\lambda$ are closed in $M$.
Then, for a sufficiently small $\varepsilon > 0$,
the cardinality of $\lambda$
such that $A_\lambda$ intersects
$\overline{f^{-1}([c - \varepsilon,
c))} \cup \overline{f^{-1}((c, c+\varepsilon])}$
is at most finite by virtue of Lemma~\ref{lem1}
and Remark~\ref{rem:lem1}.

On the other hand, suppose $A_\lambda$
does not intersect $\overline{f^{-1}([c - \varepsilon,
c))} \cup \overline{f^{-1}((c, c+\varepsilon])}$
for some $\lambda \in \Lambda$.
Then, we have
$$A_\lambda \subset
f^{-1}(c) \setminus (\overline{f^{-1}([c - \varepsilon,
c))} \cup \overline{f^{-1}((c, c+\varepsilon])}).$$
For every $x \in A_\lambda$, there
exists a small open disk neighborhood $U_x$ of $M$
such that $U_x \cap (f^{-1}([c - \varepsilon,
c)) \cup f^{-1}((c, c+\varepsilon])) = \emptyset$.
This implies that $U$ is completely contained in $f^{-1}(c)$
by the continuity of $f$. Then, we have $U_x \subset A_\lambda$.
This shows that $A_\lambda$
is an open codimension zero submanifold of $M$. As
it is also closed and connected, $A_\lambda$
is a component of $M$. Thus, the cardinality of
such $\lambda$'s is at most finite.

This completes the proof of Lemma~\ref{lem2}.
\end{proof}

\begin{proof}[Proof of Theorem~\textup{\ref{thm:Reeb}}]
Let $C_f \subset \R$ be the set of critical values
of $f$, which is finite by our assumption.
Let $G_f$ be the graph constructed as follows.
The vertices correspond bijectively to the connected components
of $f^{-1}(C_f)$ that contain critical points. By our assumption
and Lemma~\ref{lem2}, the number of vertices is finite.
Let $L_f \subset f^{-1}(C_f)$ denote the union of such
connected components. 
Then, the edges correspond bijectively to the connected components
of $M \setminus L_f$. Note that each such connected component
is diffeomorphic to the product of a closed connected
$(m-1)$--dimensional manifold and an open interval.
For each vertex $v$ (or edge $e$), let us denote by
$V_v$ (resp.\ $E_e$) the component of
$L_f$ (resp.\ $M \setminus L_f$) corresponding to $v$
(resp.\ $e$). 
Then, an edge $e$ is incident to a vertex $v$ if
and only if the closure of $E_e$
intersects $V_v$. More precisely, each edge $e$
is oriented, and its initial vertex (or the
terminal vertex) is given by $v$
if and only if $x > f(V_v)$ (resp.\ $x < f(V_v)$)
for all $x \in f(E_e)$
and the closure of $E_e$ intersects $V_v$.
Note that by the proof of Lemma~\ref{lem1}, this is
well-defined and we get a finite graph $G_f$.

In the following, the terminology ``edge''
often refers to the corresponding open $1$--cell
of the graph.
For each edge $e$, 
$f(E_e)$ is an open interval, say $(a_e, b_e)$ for $a_e < b_e$.
Then, we have a canonical orientation preserving
embedding $h_e : e \to (a_e, b_e) \subset
\R$. These functions for all $e$ can naturally be
extended to a continuous function $h : G_f \to \R$,
where $h(v) = f(V_v)$ for each vertex $v$.
(The reader is referred to the commutative
diagram in Figure~\ref{fig6} for various spaces and maps
defined here and in the following.) 

\begin{figure}[h]
\centering
\psfrag{f}{$f$}
\psfrag{M}{$M$}
\psfrag{R}{$\R$}
\psfrag{rho}{$\rho$}
\psfrag{Rf}{$W_f$}
\psfrag{Gf}{$G_f$}
\psfrag{q}{$q_f$}
\psfrag{Q}{$Q_f$}
\psfrag{h}{$h$}
\psfrag{bf}{$\bar{f}$}
\includegraphics[width=0.9\linewidth,height=0.3\textheight,
keepaspectratio]{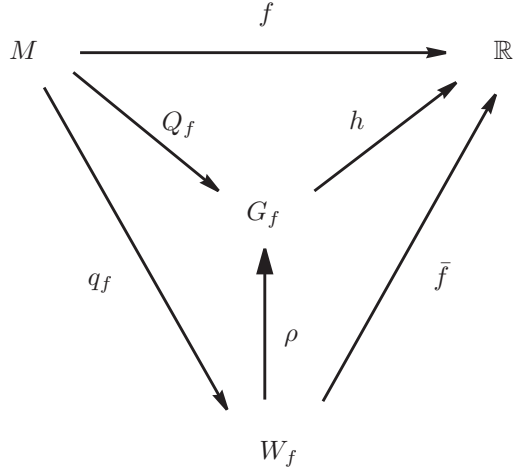}
\caption{A commutative diagram}
\label{fig6}
\end{figure}

Let $p_e : E_e \to e$ be the composition
of $f|_{E_e} : E_e \to (a_e, b_e)$ and
$h_e^{-1}$. Then, these maps $p_e$
for all the edges $e$ can naturally be extended
to a map $Q_f : M \to G_f$ in such a way that
$Q_f(V_v) = v$ for all vertices $v$ of $G_f$.
By our definition of $G_f$,
this is well defined and continuous.

Let us define the map $\rho : W_f \to G_f$ as follows.
For a point $y \in W_f$, $q_f^{-1}(y)$ is contained
in a unique $E_e$ or $V_v$. Then, $\rho(y)$
is defined to be the point in $e$ corresponding to $p_e(q_f^{-1}(y))$
in the former case, and is defined to be $v$
in the latter case. Since $Q_f = \rho \circ q_f$
is continuous, we see that $\rho$ is continuous.
Furthermore, we see easily that $\rho$
is bijective. Since $W_f$ is compact and $G_f$
is Hausdorff, we conclude that $\rho$
is a homeomorphism. Thus, $W_f$
has the structure of a graph.

Finally, we see easily that
$\bar{f} = h \circ \rho : W_f \to \R$.
As $h$ is an embedding on each edge, so
is $\bar{f}$. This completes the proof.
\end{proof}

\begin{rem}
By the above proof, for the graph structure of $W_f$,
the set of vertices of $W_f$ corresponds bijectively
to the set of connected components of level sets
containing critical points.
\end{rem}

Let us give an example of a smooth
function with \emph{infinitely many critical
values} for which the consequence of Theorem~\ref{thm:Reeb}
does not hold.

\begin{ex}
Let $M$ be an arbitrary smooth closed manifold
of dimension $m \geq 2$.
Then, by \cite{Wh1934}, there always
exists a smooth function $f : M \to [0, \infty)$
such that $f^{-1}(0)$ is a Cantor set
embedded in $M$.
In particular, $f^{-1}(0)$ has uncountably many connected
components.
Thus, the consequence of Theorem~\ref{thm:Reeb} does not hold for
such an $f$.
In this example, we can show that $f$ has
infinitely many critical values.
%

Let us give a more explicit example.
Let $\varphi_1 : \R^2 \to [0, 1]$ be a smooth
function as follows (see Figure~\ref{fig7}). 
\begin{enumerate}
\item The level set $\varphi_1^{-1}(0)$ coincides with the complement
of the open unit disk.
\item The level set $\varphi_1^{-1}(1)$ coincides with the disjoint
union of two disks centered at $a_1 = (-1/2, 0)$ and
$a_2 = (1/2, 0)$ with radius $(1/4) - \varepsilon$ 
for a sufficiently small $\varepsilon > 0$.
\item On $\varphi_1^{-1}((0, 1))$, it has a unique
critical point at the origin, which is non-degenerate of
index $1$ and whose $\varphi_1$--value is equal to $1/2$.
\item The level set $\varphi_1^{-1}(t)$
is homeomorphic to a circle, which is connected, 
for all $t \in (0, 1/2)$, and is homeomorphic to
the union of two circles for all $t \in (1/2, 1)$.
In particular, for $t$ sufficiently close to $0$ or $1$,
the components of the level sets are circles
whose centers are the origin and the points $a_1$ and $a_2$,
respectively.
\end{enumerate}

\begin{figure}[h]
\centering
\psfrag{1}{$1$}
\psfrag{0}{$0$}
\psfrag{h}{$\varphi_1 = \frac{1}{2}$}
\psfrag{f1}{$\varphi_1 = 1$}
\psfrag{f0}{$\varphi_1 = 0$}
\psfrag{a1}{$a_1$}
\psfrag{a2}{$a_2$}
\includegraphics[width=\linewidth,height=0.4\textheight,
keepaspectratio]{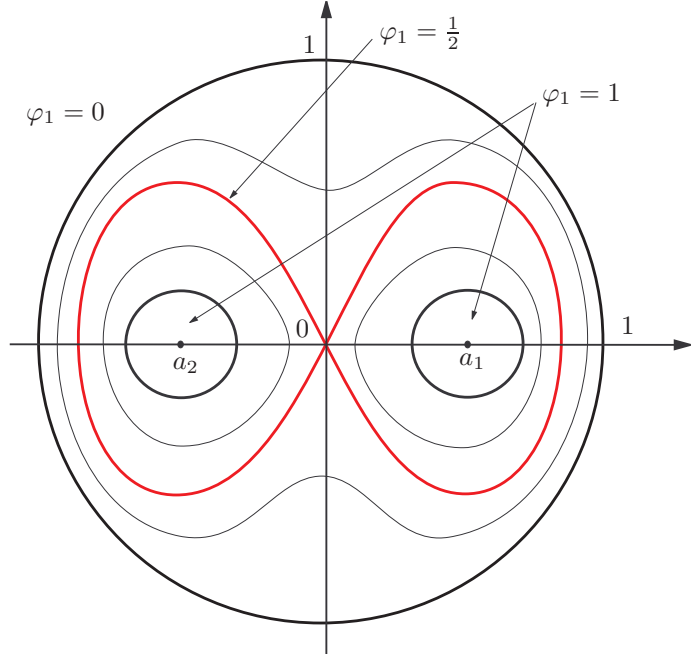}
\caption{Level sets of the smooth function $\varphi_1$}
\label{fig7}
\end{figure}

Then, we define the smooth function $\varphi_2 : \R^2 \to \R$ by
$$\varphi_2(x) = \varphi_1(4(x-a_1)) +
\varphi_1(4(x-a_2)), \, x \in \R^2.$$
Note that $\varphi_2^{-1}(1)$ consists of disjoint
four disks of radius slightly smaller than $1/16$.
Let $b_1, b_2, b_3$ and $b_4$ be the centers of the disks.
Then, we define $\varphi_3 : \R^2 \to \R$ by
$$\varphi_3(x) = \sum_{i=1}^4 \varphi_1(16(x-b_i)), 
\, x \in \R^2.$$

Repeating this procedure inductively, we can construct
a sequence of smooth functions $\varphi_n$, $n \geq 1$. Then 
consider the series
$\psi = \sum_{n=1}^\infty c_n\varphi_n$ 
for a rapidly decreasing sequence $\{c_n\}_{n=1}^\infty$ 
of positive real numbers. 
We can show that this series converges uniformly
and that $\psi$ defines a smooth function on $\R^2$.
Furthermore, we see that it has the following
properties.
\begin{enumerate}
\item The level set $\psi^{-1}(0)$ coincides with the complement
of the open unit disk.
\item The level set $\psi^{-1}(c)$ is a Cantor set, where
$c = \sum_{n=1}^\infty c_n$.
\item The critical value set of $\psi$ consists
of countably many real numbers $\{r_k\}_{k=0}^\infty$ with
$0 = r_0 < r_1 < r_2< \cdots$ converging to $c$ together with $r_\infty = c$
itself, where $r_k = \sum_{n=1}^k c_n$ for $k \geq 1$.
\item For $t \in (r_k, r_{k+1})$, the level set
$\psi^{-1}(t)$ is homeomorphic to the disjoint union
of $2^k$ circles for each $k \geq 0$.  
\end{enumerate}
Then, we see that the Reeb space of $\psi$ is as depicted
in Figure~\ref{fig8}. It consists of countably many ``edges''
and uncountably many ``vertices''. However,
this is not a cell complex, as every point of 
$\bar{\psi}^{-1}(c)$ is an
accumulation point. (In fact, one can show that
the Reeb space $W_\psi$ can be embedded in $\R^2$
as in Figure~\ref{fig8}.)

\begin{figure}[h]
\centering
\psfrag{0}{$r_0 = 0$}
\psfrag{b}{$\bar{\psi}$}
\psfrag{r1}{$r_1$}
\psfrag{r2}{$r_2$}
\psfrag{r3}{$r_3$}
\psfrag{W}{$W_\psi$}
\psfrag{R}{$\R$}
\psfrag{C}{Cantor set}
\psfrag{3}{$c$}
\includegraphics[width=0.9\linewidth,height=0.3\textheight,
keepaspectratio]{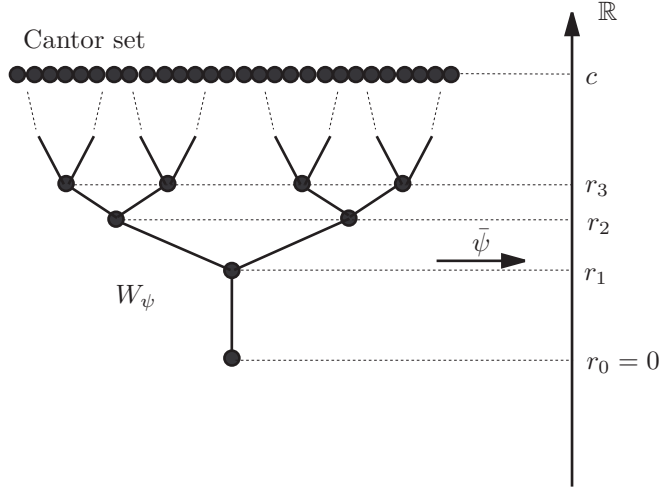}
\caption{Reeb space of $\psi$}
\label{fig8}
\end{figure}

Note that, for an arbitrary closed connected surface $M$,
by embedding $\R^2$ into $M$ and by extending the function
$\psi$ by the zero function on the complement, we can
construct a smooth function $f : M \to \R$ whose
Reeb space has the same properties.
\end{ex}

Let us give another example of a smooth
function whose Reeb space does not have
the structure of a graph.

\begin{ex}
Let us consider the smooth function $f : \R^2 \to \R$
as follows. For $n \geq 1$, let $D_n$ be the closed disk
in $\R^2$ centered at the point $(1/n, 0)$ with radius
$1/2n(n+1)$. Note that the disks $\{D_n\}_{n=1}^\infty$ 
are disjoint.
Let $g : D^2 \to [0, 1]$ be a smooth function on the unit disk
in $\R^2$ with the following properties.
\begin{enumerate}
\item The function $g$ restricted to a small 
collar neighborhood $C(\partial D^2)$ 
of $\partial D^2$ is constantly zero.
\item The restriction $g|_{D^2 \setminus C(\partial D^2)}$ 
has the unique critical point at the origin,
which is the maximum point and takes the value $1$.
\item Each level set $g^{-1}(t)$ is a circle centered
at the origin for $t \in (0, 1)$.
\end{enumerate}
Note that the Reeb space of $g$ can be identified with $[0, 1]$.
Let $f_n$ be the smooth function on $D_n$ 
defined by
$$f_n(x) = g\left(2n(n+1)\left(x-\frac{1}{n}\right)\right).$$
Then, we define $f : \R^2 \to \R$ by
$f(x) = c_nf_n(x)$ if $x \in D_n$ and $f(x) = 0$ otherwise
for a rapidly decreasing sequence $\{c_n\}_{n=1}^\infty$
of positive real numbers.

We can show that $f$ is a smooth function.
Furthermore, we can also show that 
the Reeb space $W_f$ of $f$ is homeomorphic to
the union of line segments $I_n$ in $\R^2$, $n \geq 1$, where
$$I_n = \left\{t\left(\cos \frac{\pi}{n}, \sin \frac{\pi}{n}\right) 
\in \R^2\,\left|\, 0 \leq t \leq c_n\right.\right\}.$$
So, $W_f$ is a union of infinitely many intervals where
exactly one end point of each interval is glued to a fixed point, say $p$
(see Figure~\ref{fig9}).
In this sense, it seems to have the structure of a cell complex.
However, its possible vertex set $V$ does not have the discrete
topology. In fact, $p$ is in the closure of $V \setminus \{p\}$.
Therefore, $W_f$ does not have the topology of a graph.

\begin{figure}[h]
\centering
\psfrag{c1}{$c_1$}
\psfrag{c2}{$c_2$}
\psfrag{0}{$0$}
\psfrag{p}{$p$}
\psfrag{b}{$\bar{f}$}
\psfrag{W}{$W_f$}
\psfrag{R}{$\R$}
\includegraphics[width=0.9\linewidth,height=0.3\textheight,
keepaspectratio]{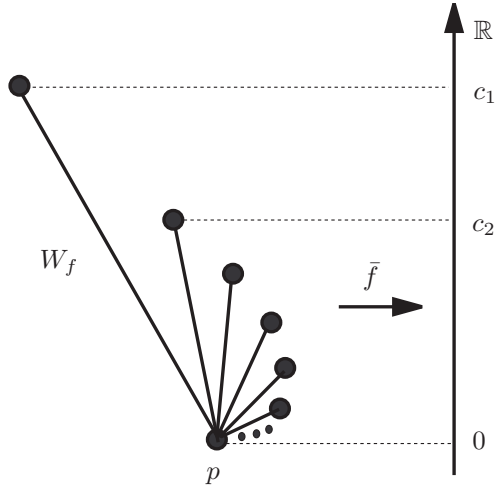}
\caption{Reeb space of $f$}
\label{fig9}
\end{figure}

Note that, in this example, the critical value set
is equal to the infinite set 
$$\{c_n \,|\, n = 1, 2, 3, \ldots\} \cup \{0\}.$$

Note also that by embedding $\R^2$ into an arbitrary
closed surface $M$ and by extending the function
by the zero map in the complement, one can construct
a smooth function $M \to \R$ with
Reeb space having the same property.
Note also that for this example, an arbitrary level
set has at most finitely many connected components;
however, the number of connected components is not
uniformly bounded from above.
\end{ex}

Let us give an example of a non-triangulable
smooth function with finitely many critical values.

\begin{ex}
Given an arbitrary smooth closed manifold
$M$ of dimension $m \geq 2$, there always exists a smooth function
$f : M \to \R$ with finitely many
critical values that is not triangulable.
Nevertheless, even in such a situation, 
$W_f$ is a graph.

In fact, we can construct such a function
so that for a critical value $c$, the
closed set $f^{-1}(c)$
cannot be triangulated as follows.
First, take any smooth function $g : M \to \R$
with finitely many critical values
and a regular value $c$. Then, modify $M$
along the submanifold $L = g^{-1}(c)$ as follows.
Since the set of critical values is closed in $\R$,
the closed interval $I = [c-\varepsilon, c+\varepsilon]$
contains no critical values for some $\varepsilon > 0$.
Note that then $g^{-1}(I) \cong L \times I$. In the
following, we fix such a diffeomorphism and identify
$g^{-1}(I)$ with $L \times I$.

On the other hand, we consider smooth functions
$h_i : L \to I$, $i=1, 2$,
such that
\begin{enumerate}
\item $h_1(x) \leq h_2(x)$ for all $x \in L$, and
\item the set
$$\tilde{L} = \{(x, t) \in L \times I \,|\,
x \in L, h_1(x) \leq t \leq h_2(x)\}$$
is not triangulable.
\end{enumerate}
(For example, construct such functions in such a way
that the interior of $\tilde{L}$
in the compact set $L \times I$
has infinitely many connected components.)

Then, we consider the compact manifold
$M \setminus g^{-1}((c-\varepsilon, c+\varepsilon))$
and glue $\tilde{L}$ along the boundary in such a way that 
$(x, h_1(x))$ (resp.\ $(x, h_2(x))$) in $\tilde{L}$
is identified with $(x, c-\varepsilon)$ (resp.\ $(x, c+\varepsilon)$)
in $g^{-1}(c - \varepsilon) \cong L \times \{c-\varepsilon\}$
(resp.\ $g^{-1}(c + \varepsilon) \cong L \times \{c+\varepsilon\}$).
Then, the resulting space $\tilde{M}$ is easily seen
to be a smooth manifold naturally diffeomorphic to $M$.
Furthermore, we can modify $g$ slightly on a neighborhood
of $g^{-1}(I)$ by using bump functions
to get a smooth function $f : \tilde{M} \to \R$
such that $f^{-1}(c) = \tilde{L}$ and $c$ is the unique
critical value of $f$ in $I$.
Then, the smooth function $f$ has the desired
properties.
See Figure~\ref{fig2}.
\end{ex}

\begin{figure}[h]
\centering
\psfrag{f}{$f$}
\psfrag{M}{$\tilde{M}$}
\psfrag{r}{$c$}
\psfrag{S}{$\tilde{L}$}
\includegraphics[width=\linewidth,height=0.4\textheight,
keepaspectratio]{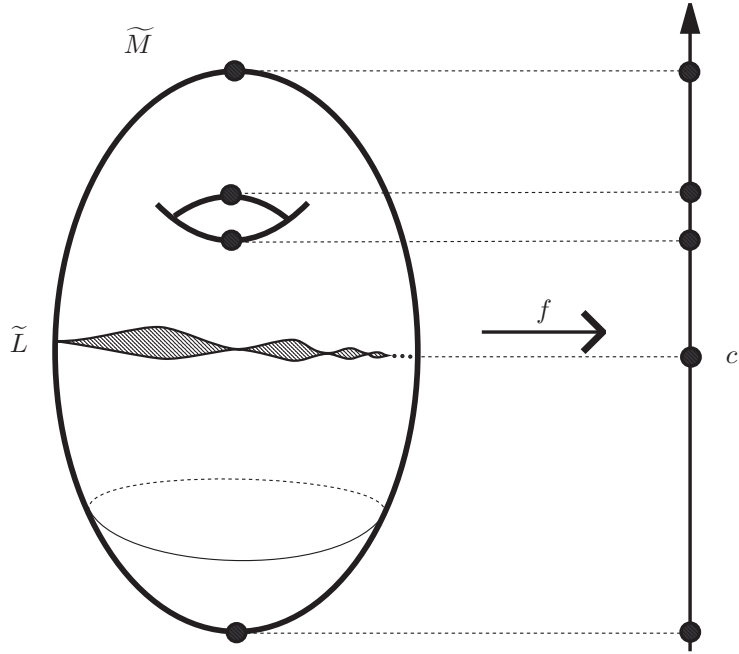}
\caption{Non-triangulable function with finitely many critical values}
\label{fig2}
\end{figure}

\section{Path Reeb spaces}\label{section3.5}

In the definition of a Reeb space, we usually
use connected components of level sets.
If we use \emph{path-components} instead, then
the resulting space is called the \emph{path Reeb space}.

\begin{dfn}
(1) Let $X$ be a topological space.
A continuous map $\lambda : [0, 1] \to X$ is called a
\emph{path} connecting $\lambda(0)$ and $\lambda(1) \in X$.
We say that two points $x_0, x_1 \in X$ are
\emph{connected by a path} 
if there exists a path connecting $x_0$ and $x_1$.
This defines an equivalence relation on $X$ and each equivalence
class is called a \emph{path-component} of $X$.
A topological space $X$ is said to be \emph{path-connected}
if it consists of a unique path-component.

(2) Let $f : X \to Y$ be a continuous map between topological spaces.
For two points $x_0, x_1 \in X$, we define $x_0 \sim_\mathrm{p} x_1$
if $f(x_0) = f(x_1)$ and $x_0, x_1$ lie on the same path-component
of $f^{-1}(f(x_0)) = f^{-1}(f(x_1))$.
Let $W^\mathrm{p}_f = X/\!\sim_\mathrm{p}$ be the quotient space 
with respect to this equivalence relation: i.e.\ $W^\mathrm{p}_f$
is a topological space endowed with the quotient topology.
Let $q^\mathrm{p}_f : X \to W^\mathrm{p}_f$ denote the quotient map.
Then, there exists a unique map $\bar{f}^\mathrm{p} : W^\mathrm{p}_f \to Y$ that 
is continuous and makes the 
following diagram commutative:
\begin{eqnarray*}
X \!\!\!\! & \spmapright{f} & \!\!\!\! Y \\
& {}_{q^\mathrm{p}_f}\!\!\searrow \quad \qquad \nearrow_{\bar{f}^\mathrm{p}} & \\
& \,W^\mathrm{p}_f. &
\end{eqnarray*}
The space $W^\mathrm{p}_f$ is called the
\emph{path Reeb space} of $f$, and
the map $\bar{f}^\mathrm{p} : W^\mathrm{p}_f \to Y$ is called
the \emph{path Reeb map} of $f$. The decomposition
of $f$ as $\bar{f}^\mathrm{p} \circ q^\mathrm{p}_f$ as in the above
commutative diagram is called the \emph{path Stein
factorization} of $f$. 
\end{dfn}

The path Reeb space and the usual Reeb space are not homeomorphic
to each other in general. To see this, let us first observe the following.

\begin{lem}
Let $f : X \to Y$ be a continuous map
between topological spaces.
If $Y$ is a $T_1$--space, then so is the
Reeb space $W_f$.
\end{lem}

\begin{proof}
Take a point $c \in W_f$. Then, $q_f^{-1}(c)$
is a connected component of a level set
$f^{-1}(\bar{f}(c))$. Since $Y$ is a $T_1$--space
and $f$ is continuous, $f^{-1}(\bar{f}(c))$ is closed
in $X$. Therefore, each of its connected components
is closed in $X$, and in particular, $q_f^{-1}(c)$
is closed in $X$. Hence, $\{c\}$ is closed in $W_f$.
\end{proof}

For example, for an arbitrary
smooth closed manifold $M$ of dimension $m \geq 2$,
we can easily
construct a compact set $K$ in $M$ which is connected,
but is not path-connected. There exists a smooth
function $f : M \to [0, 1]$ such that $f^{-1}(0) = K$.
Then, the Reeb space $W_f$ and the path Reeb space
$W^\mathrm{p}_f$ are not the same.
Furthermore, we can find a compact connected set $K$
as above containing a path-component
that is not compact. Then, the path Reeb space
of the resulting smooth
function $f$ is not a $T_1$--space,
since the point in the path Reeb space $W^\mathrm{p}_f$
corresponding to
that path-component is not closed.
So, the Reeb space $W_f$ and the path Reeb space $W^\mathrm{p}_f$ are
not homeomorphic to each other.


Let us give an example of a smooth function
on a closed manifold with finitely many critical values
whose path Reeb space is not a $T_1$--space.

\begin{ex}
Let $\Delta_1$ and $\Delta_2$ be closed $2$--disks
disjointly embedded in the $2$--sphere $S^2$.
Let $C$ be a curve homeomorphic to the real line
$\R$ embedded in the annular region $S^2 \setminus (\Delta_1 \cup \Delta_2)$
whose ends wind around $\partial \Delta_1$ and $\partial \Delta_2$
as in Figure~\ref{fig10}. Note that $C$ is not closed in $S^2$
and the closure $\overline{C}$ coincides with $C \cup \partial \Delta_1
\cup \partial \Delta_2$. Note also that $K = C \cup \Delta_1 
\cup \Delta_2$ is a
compact connected set which is not path-connected.

\begin{figure}[h]
\centering
\psfrag{S}{$S^2$}
\psfrag{d1}{$\Delta_1$}
\psfrag{d2}{$\Delta_2$}
\psfrag{c}{$C$}
\includegraphics[width=\linewidth,height=0.5\textheight,
keepaspectratio]{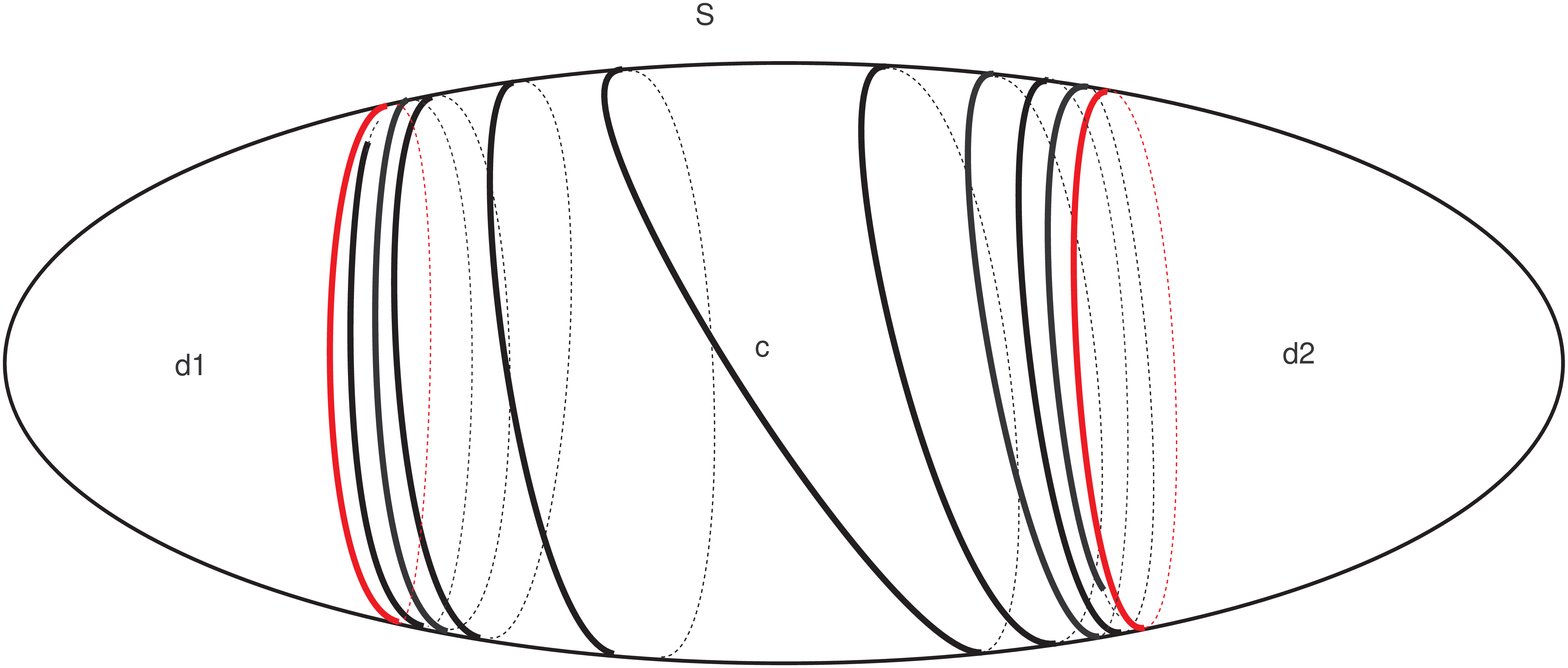}
\caption{Zero level set of $f : S^2 \to \R$}
\label{fig10}
\end{figure}

We see easily that $S^2 \setminus K$ is diffeomorphic
to an open $2$--disk. Let us fix a diffeomorphism
$\varphi : \Int{D^2} \to S^2 \setminus K$, where
$\Int{D^2}$ is the open unit disk in $\R^2$.
For $n \geq 1$, let $g_n: \Int{D^2} \to [0, 1]$ be a smooth
function with the following properties.
\begin{enumerate}
\item The zero level set $g_n^{-1}(0)$ coincides with
the complement of the open disk $D_n$
of radius $1-(n+1)^{-1}$ centered at the origin.
\item $g_n|_{D_n \setminus \{0\}}$ is a smooth
function of $r = \sqrt{x^2+y^2} \in (0, 1-(n+1)^{-1})$, say
$g_n(x, y) = h_n(r)$, with $h_n'(r) < 0$.
\item $g_n|_{D_n}$ has a unique critical point at the
origin, which is the maximum point of $g_n$, with $g_n(0) = 1$.
\end{enumerate}
Let $f_n : S^2 \to \R$ be the function defined by
$f_n(z) = g_n \circ \varphi^{-1}(z)$ for $z \in S^2 \setminus K$,
and $f_n(z) = 0$ otherwise. We see easily that $f_n$ defines
a smooth function. Then, we set
$f = \sum_{n=1}^\infty c_n f_n$ for a rapidly decreasing
sequence $\{c_n\}_{n=1}^\infty$ of positive real numbers,
so that $f$ defines a smooth function $f : S^2 \to \R$.

By construction, we see that $f^{-1}(0) = K$. Since $f$
takes the minimum value $0$ on $K$, all the points
in $K$ are critical points. Furthermore,
$f$ has a unique critical point, $\varphi(0)$, on $S^2 \setminus K$,
which is the maximum point of $f$.
Hence, the critical value set consists of two values: $0$ and
$f(\varphi(0)) = \sum_{n=1}^\infty c_n$.

The Reeb space $W_f$ is easily seen to be homeomorphic to
a closed interval, which has the structure of a graph
(see Figure~\ref{fig11}). On the other hand, the
path Reeb space $W^\mathrm{p}_f$ is not a $T_1$--space.
More precisely, it consists of a half open interval together with three points
$v_1, v_2$ and $v_C$ that correspond to $\Delta_1$, $\Delta_2$
and $C$, respectively: $\{v_1\}$ and $\{v_2\}$ are closed, while
$\{v_C\}$ is not closed and its closure coincides with $\{v_C, v_1, v_2\}$.
Furthermore, the closure of the half open interval contains these three
points (see Figure~\ref{fig11}).

\begin{figure}[h]
\centering
\psfrag{W1}{$W_f$}
\psfrag{W2}{$W^\mathrm{p}_f$}
\psfrag{p1}{$v_1$}
\psfrag{p2}{$v_C$}
\psfrag{p3}{$v_2$}
\includegraphics[width=0.9\linewidth,height=0.25\textheight,
keepaspectratio]{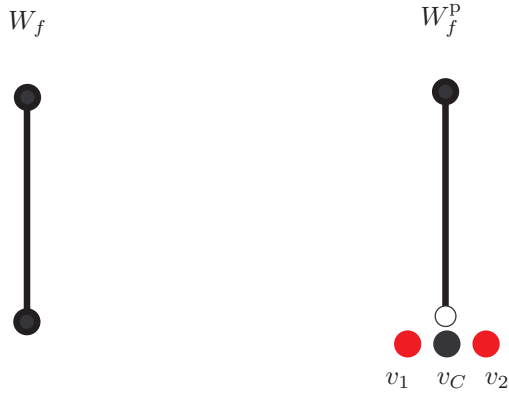}
\caption{The Reeb space and the path Reeb space of $f$}
\label{fig11}
\end{figure}

This example shows that our Theorem~\ref{thm:Reeb}
does not hold in general for path Reeb spaces.
In fact, Lemma~\ref{lem1} does not hold for
the numbers of path-components for this
example.
\end{ex}

However, for smooth functions with finitely many \emph{critical
points}, we have the following.

\begin{prop}
Let $f : M \to \R$ be a smooth function
on a closed manifold with finitely many critical points.
Then, we have $W_f = W^\mathrm{p}_f$.
\end{prop}

\begin{proof}
Take $y \in \R$. If $y$ is a regular value of $f$,
then $f^{-1}(y)$ is a smooth manifold and its
connected components and path-components coincide.
If $y$ is a critical value, then take a sufficiently
small positive real number $\varepsilon$ such that
$y$ is the unique critical value of $f$ in 
$[y - \varepsilon, y + \varepsilon]$.
Then, by an argument using integral curves of
a gradient vector field of $f$ together with our
assumption that $f$ has only finitely many critical
points, we can show that $f^{-1}(y)$
is a deformation retract of $V = f^{-1}([y - \varepsilon,
y + \varepsilon])$.
As $V$ is a compact manifold with boundary and
is locally path-connected, we see that the
connected components and the path-components
coincide with each other and the numbers are
finite. As $f^{-1}(y)$ is homotopy equivalent
to $V$, the number of connected components
(or the number of path-components)
coincides with that of $V$.
This implies that the number of connected components and
that of path-components also coincide with each other
for $f^{-1}(y)$. As the numbers are finite, we see that
the connected components and the path-components coincide
with each other for $f^{-1}(y)$.
This completes the proof.
\end{proof}

The above example shows that this proposition
is no longer true for smooth functions with
finitely many critical values in general.

\section{Realization I}\label{section4}

Let $G$ be a graph without loops.
In this section, we prove that $G$ is 
always realized as the
Reeb space of a certain smooth function 
on a closed manifold with
finitely many critical values. In fact, our
theorem is stronger: (fat) level sets can
also be preassigned.

\begin{dfn}
For $m \geq 2$, consider maps
\begin{eqnarray*}
& & \Phi : \{\text{edges of $G$}\}
\to \{\text{diffeomorphism types of
closed connected} \\
& & \qquad \qquad \qquad \qquad \qquad \qquad
\qquad \qquad \qquad
\text{$(m-1)$--dimensional manifolds}\}, \\
& & \Gamma : \{\text{vertices of $G$}\}
\to \{\text{diffeomorphism types of
compact connected} \\
& & \qquad \qquad \qquad \qquad \qquad \qquad
\qquad \qquad \qquad
\text{$m$--dimensional manifolds}\},
\end{eqnarray*}
such that for each vertex $v$ of $G$, we have
\begin{equation}
\partial (\Gamma(v)) \cong
\sqcup_{v \in e} \Phi(e),
\label{cob}
\end{equation}
i.e.\ the boundary of $\Gamma(v)$ is
diffeomorphic to the disjoint union of the finitely
many manifolds $\Phi(e)$, where $e$ runs over all
edges incident to $v$. The triple $(G, \Phi, \Gamma)$
is called an \emph{$m$--decorated graph}.
\end{dfn}

\begin{dfn}
We say that an $m$--decorated graph $(G, \Phi, \Gamma)$ is
\emph{Reeb realizable}
if there exists a smooth function $f : M
\to \R$ on a closed $m$--dimensional
manifold $M$ with finitely many critical values
such that
\begin{enumerate}
\item $G$ is identified with the Reeb space $W_f$,
\item for each edge $e$ of $G = W_f$, we have
$q_f^{-1}(x) \cong \Phi(e)$, $\forall x \in \Int{e}$, 
\item for each vertex $v$ of $G = W_f$, we have
$q_f^{-1}(N(v)) \cong \Gamma(v)$, where $N(v)$ is
a small regular neighborhood of $v$ in $G$.
\end{enumerate}
See Figure~\ref{fig3}.
\end{dfn}

\begin{figure}[h]
\centering
\psfrag{e}{$e$}
\psfrag{q}{$q_f$}
\psfrag{f}{$f$}
\psfrag{G}{$\Gamma(v) \subset M$}
\psfrag{v}{$v$}
\psfrag{F}{$\Phi(e)$}
\psfrag{W}{$W_f = G$}
\psfrag{b}{$\bar{f}$}
\psfrag{N}{$N(v)$}
\includegraphics[width=0.9\linewidth,height=0.7\textheight,
keepaspectratio]{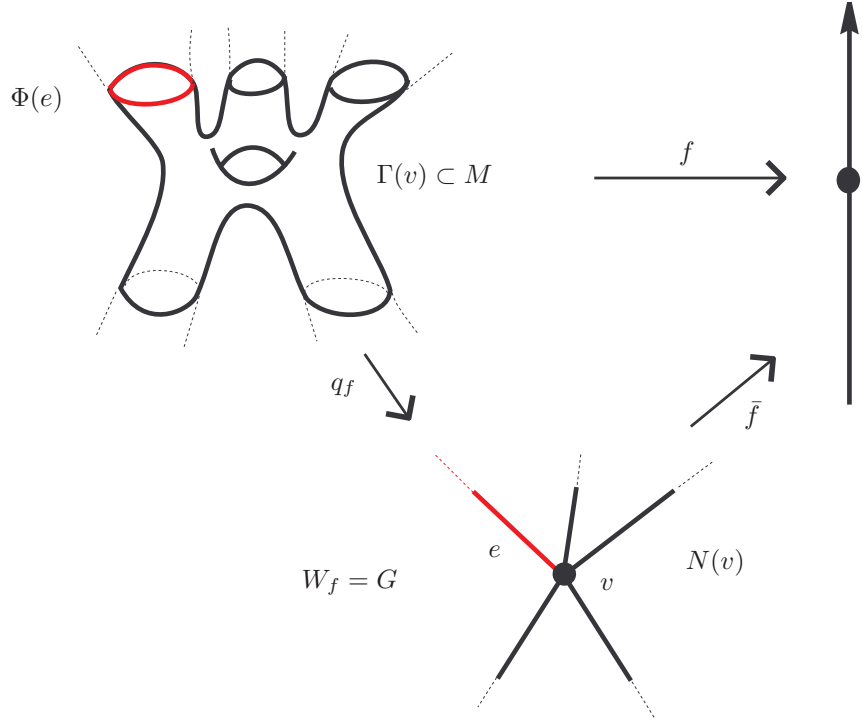}
\caption{Realizing an $m$--decorated graph $(G, \Phi, \Gamma)$, where
$\dim{M} = m$}
\label{fig3}
\end{figure}

Then, we have the following realization theorem.

\begin{thm}\label{thm2}
For $m \geq 2$,
every $m$--decorated graph $(G, \Phi, \Gamma)$ is Reeb realizable.
\end{thm}

\begin{proof}
Let us first construct a continuous function 
$h : G \to \R$ such that it is an embedding
on each edge. Such a function can easily be
constructed by first defining $h$ on the set
of vertices so that it is injective, and then
by extending it on the closure of each edge
``linearly''.

By $h$, we can identify each edge of $G$ with
an open interval of $\R$.

Now, for each vertex $v$, define $f_v : \Gamma(v) \to \R$
as a constant function by $f_v(\Gamma(v)) = h(v)$.
For each edge $e$, identified with a bounded open
interval $(a, b)$, consider a smooth function $\varphi_e : e \to e$
such that
\begin{enumerate}
\item $\varphi_e$ is a monotone increasing diffeomorphism,
\item $\varphi_e$ can be extended to a smooth homeomorphism
$[a, b] \to [a, b]$ such that the $r$--th derivatives of $f$
at $a$ and $b$ all vanish for all $r \geq 1$.
\end{enumerate}
Such a smooth function can be constructed, for example,
by integrating a bump function.
Then, for each vertex $e$, define 
the smooth function $f_e : \Phi(e) \times e \to \R$
by $h \circ \varphi_e \circ p_e$, where $p_e : \Phi(e) \times
e \to e$ is the projection to the second factor.

Then, by condition (\ref{cob}), we can glue
the $m$--dimensional manifolds $\Gamma(v)$ and
$\Phi(e) \times e$ over all vertices $v$ and all
edges $e$ of $G$ in such a way that 
$\partial (\Gamma(v))$ and 
$\sqcup_{v \in e} \Phi(e)$ are identified by
diffeomorphisms. Let us denote by $M$
the resulting closed $m$--dimensional manifold.
Then, we can also glue the smooth functions $f_v$ and $f_e$
over all vertices $v$ and all edges $e$ of $G$
in order to obtain a continuous function $f : M \to \R$.
This function $f$ is smooth by construction.

Then, by construction and by our assumption
that $\Gamma(v)$ and $\Phi(e)$ are connected, 
the Reeb space $W_f$ of $f$ can be identified with $G$.
Furthermore, we see easily that
for each edge $e$ of $G = W_f$, we have
$q_f^{-1}(x) \cong \Phi(e)$, $\forall x \in \Int{e}$, 
and for each vertex $v$ of $G = W_f$, we have
$q_f^{-1}(N(v)) \cong \Gamma(v)$ for a small regular
neighborhood $N(v)$ of $v$ in $G$.
This completes the proof.
\end{proof}

\begin{rem}
We have a similar theorem for functions on
compact manifolds with boundary as well. In this case,
$\Phi$ associates to each edge of $G$ a compact connected
$(m-1)$--dimensional manifold possibly with boundary, 
and $\Gamma$ associates
to each vertex of $G$ a compact connected $m$--dimensional
manifold with corners. Condition (\ref{cob})
should appropriately be modified.
\end{rem}

\begin{cor}
For all $m \geq 2$,
every graph without
loops is the Reeb space of a smooth
function on a closed $m$--dimensional 
manifold with finitely many critical values.
\end{cor}

\begin{rem}\label{rem1}
In the above corollary, we can even construct
such a smooth function in such a way that
\begin{enumerate}
\item every regular level set is a finite
disjoint union of standard $(m-1)$--spheres, and
\item the source manifold is diffeomorphic to
$S^m$ or a connected sum of a finite number of
copies of $S^1 \times S^{m-1}$.
\end{enumerate}
This can be achieved by associating $S^{m-1}$
to each edge, $S^m \setminus (\sqcup_{k=1}^d \Int{D^m_k})$
to each vertex with degree $d$, and by choosing
the attaching diffeomorphisms appropriately, where
$D^m_k$, $k = 1, 2, \ldots, d$, are disjoint
$m$--dimensional disks embedded in $S^m$.
\end{rem}

\begin{rem}
Some results similar to Theorem~\ref{thm2} and Remark~\ref{rem1}
are presented in \cite{K1, K2}.
\end{rem}

\section{Realization II}\label{section5}

Let $f : M \to \R$ be a smooth function
on a closed manifold of dimension $m \geq 2$
with finitely many critical values.
Then, it is easy to show that the homomorphism
$(q_f)_* : \pi_1(M) \to \pi_1(W_f)$ induced by the
quotient map is surjective. See 
\cite[Th\'{e}or\`{e}me~6]{R} for a related statement.
(In fact,
this is true for an arbitrary continuous function $f$
as long as $W_f$ is semilocally
simply-connected. See \cite{CGM}.)

We have the following theorem, which corresponds to
the converse of this fact.

\begin{thm}\label{thm3}
Let $M$ be a smooth closed connected manifold
of dimension $m \geq 2$,
$G$ a connected graph
without loops, and $Q : M \to G$ a
continuous map such that $Q_* : \pi_1(M)
\to \pi_1(G)$ is surjective.
Then, there exists a smooth function 
$f : M \to \R$ with finitely many critical
values such that
\begin{enumerate}
\item $G$ can be identified with $W_f$ in such a way
that the vertices are identified with the $q_f$--images
of the level set components containing
critical points,
\item $q_f : M \to W_f = G$ is
homotopic to $Q$.
\end{enumerate}
\end{thm}

 
\begin{rem}
A similar result also
holds for functions on compact manifolds
with non-empty boundary.
\end{rem}

\begin{rem}
A similar result has been obtained by
Michalak \cite{M1, M2} (see also Gelbukh \cite{G3},
Marzantowicz and Michalak \cite{MM}).
For $m \geq 3$, one can realize
a given graph as the Reeb space of a Morse function
on a closed manifold of dimension $m$
up to homeomorphism.
Our theorem is slightly different from such results
in that we not only realize the topological
structure of a given graph but we also 
realize the given graph structure.
We construct smooth functions with finitely
many critical values such that the images by the quotient map
of the level set components containing critical points 
exactly coincide with the vertices of the graph.

For example, if $G$ is a connected graph consisting of two
vertices and a unique edge connecting them, then for
any closed connected manifold $M$ of dimension $m \geq 2$, there
exists a smooth function $f : M \to \R$ with
exactly two critical values such that $W_f$ has the
graph structure equivalent to $G$. 
Note that according to Michalak \cite{M1, M2},
there exists a Morse function $g : M \to \R$
such that $W_f$ is homeomorphic to $G$; however,
the graph structure of $W_f$ may not be equivalent to that
of $G$ in general. In fact, if $M$ admits a smooth function
$f : M \to \R$ with exactly two (possibly degenerate)
critical points, then
$M$ must be homeomorphic to the sphere $S^n$ \cite{Milnor1,
Milnor2}.
\end{rem}

\begin{proof}[Proof of Theorem~\textup{\ref{thm3}}]
Let $r$ denote the first Betti number of $G$. Then,
we can find points $x_1, x_2, \ldots, x_r$
on edges of $G$ such that $G \setminus \{x_1, x_2, \ldots,
x_r\}$ is connected and contractible.
Let $h : G \to \R$ be a continuous
map such that $h$ is an embedding on each edge.
By using $h$, we induce a differentiable structure
on each edge of $G$.
By perturbing $Q$ by homotopy if necessary,
we may assume that $Q$ is transverse to $x_1, x_2,
\ldots, x_r$.
Then, each $Q^{-1}(x_i)$ is a compact
$(m-1)$--dimensional submanifold of $M$, which
might not be connected.
Let $\tilde{r}$ denote the total number
of connected components of $\cup_{i=1}^r Q^{-1}(x_i)$.

Let us show that we may arrange, by homotopy of $Q$,
so that $\tilde{r} = r$, i.e.\ each $Q^{-1}(x_i)$ is connected.

Suppose that $\tilde{r} > r$. Then, by re-ordering the points,
we may assume that $E_1 = Q^{-1}(x_1)$ is not connected.
As $M$ is connected and $m = \dim{M}
\geq 2$, there is a smooth embedded curve $\gamma$
in $M$ whose end points lie in distinct components
of $E_1$.
Note that $Q|_{\gamma}$ is a loop in $G$ based at $x_1$.
Since $Q$ induces an epimorphism $Q_* : \pi_1(M, \tilde{x}_1)
\to \pi_1(G, x_1)$, where $\tilde{x}_1$
is the initial point of $\gamma$, we may assume
that $Q|_{\gamma}$ represents the neutral element
of $\pi_1(G, x_1)$ by adding a loop based at $\tilde{x}_1$
to $\gamma$ and by modifying it by a suitable homotopy. 
We may further assume that $\gamma$
is transverse to $Q^{-1}(x_i)$ for all $i$.

First suppose that $\gamma$ does not intersect
with $Q^{-1}(x_i)$ for $i > 1$ and that $\gamma$
intersects with $E_1 = Q^{-1}(x_1)$ only at the end points.
In this case, $Q|_{\gamma}$ starts $x_1$ in a certain, say
positive, direction of the edge on which $x_1$ lies, and
returns to $x_1$ in the reverse, say negative, direction.
Furthermore, $Q|_{\gamma}$ does not intersect
$x_i$ for $i > 1$. Thus, $Q|_{\gamma}$ is null
homotopic in $G$ relative to $x_1$ and such a homotopy
can be constructed in such a way that
the paths avoid $\{x_1, x_2, \ldots, x_r\}$
during the homotopy except at the end points
and except for the final constant path.
Let $N(\gamma) \cong \gamma \times D^{m-1}$ 
be a small tubular neighborhood of $\gamma$
in $M$ such that $N(\gamma) \cap E_1
\cong \partial \gamma \times D^{m-1}$.
We may assume that $N(\gamma)$ does not
intersect $Q^{-1}(x_i)$ for $i > 1$.
Let $N'(\gamma)$ be a smaller tubular neighborhood
corresponding to $\gamma \times D^{m-1}_{1/3}$,
where $D^{m-1}$ is the unit disk in $\R^{m-1}$
and $D^{m-1}_\rho$ is the disk with radius $\rho > 0$ with the same center.
We can first modify $Q$ by homotopy supported on $N(\gamma)$
so that $Q|_{\gamma \times \{p\}}$ coincides with $Q|_{\gamma}$
for all $p \in D^{m-1}_{2/3}$.
As $Q|_{\gamma}$ is null-homotopic in $G$ relative to $x_1$
as described above,
we may further modify $Q$ so that $Q|_{N'(\gamma)}$
is a constant map to $x_1$.

At this stage, we have $Q^{-1}(x_1) = E_1 \cup N'(\gamma)$. 
Let $E_1'$ be the smooth $(m-1)$--dimensional submanifold
of $M$ obtained from $(E_1 \setminus (\partial \gamma 
\times \Int{D^{m-1}_{1/3}}))
\cup (\gamma \times \partial D^{m-1}_{1/3})$ by smoothing 
the corner. Then, we
may further modify $Q$ by homotopy supported on a neighborhood
of $N'(\gamma)$ in such a way that $Q$ is transverse to $x_1$
and that $Q^{-1}(x_1) = E_1'$. This can be achieved by
sending the part $\gamma \times \Int{D^{m-1}_{1/3}}$ to
the negative side of $x_1$.
Then, the number of connected components of $Q^{-1}(x_1)$
decreases by $1$. As $Q^{-1}(x_i)$ for $i > 1$ stay the same,
the total number of connected components of $\cup_{i=1}^r Q^{-1}(x_i)$
decreases by $1$.

Now consider the case where
$\Int{\gamma}$ intersects
with $Q^{-1}(x_i)$ for some $i \geq 1$.
We fix a positive direction on each
edge on which $x_1, x_2, \ldots, x_r$ lie.
Let $\omega$ be the word on $x_1, x_2, \ldots, x_r$
constructed by associating $x_i$ (or $x_i^{-1}$)
every time $Q|_{\gamma}$ passes through $x_i$ 
in the positive (resp.\ negative) direction.
As $\pi_1(G, x_1)$ is a free group of rank $r$
freely generated by elements corresponding to $x_i$,
$i = 1, 2, \ldots, r$, and $Q|_{\gamma}$
represents the neutral element of $\pi_1(G, x_1)$,
we see that $x_\ell x_\ell^{-1}$ or $x_\ell^{-1}x_\ell$
appears in the word $\omega$ for some $\ell$ with
$1 \leq \ell \leq r$. 
Let $\gamma_\ell$ be the sub-arc of $\gamma$
that corresponds to that sub-word: i.e.\ $Q|_{\gamma_\ell}$
starts $x_\ell$ in a certain direction and returns
to $x_\ell$ in the reverse direction, where $Q|_{\Int{\gamma_\ell}}$
does not intersect $x_1, x_2, \ldots, x_r$.

If the end points of $\gamma_\ell$
lie on the same connected component of $Q^{-1}(x_\ell)$, then
for $\gamma$, we can replace the sub-arc $\gamma_\ell$
by a path in $Q^{-1}(x_\ell)$ connecting the end points of $\gamma_\ell$
and slightly modify it so as
to get a new smooth embedded curve $\gamma'$ such that the
corresponding word has strictly fewer letters.

On the other hand, if the end points lie on distinct
components of $Q^{-1}(x_\ell)$, then we can modify $Q$ by homotopy
as described above so that we decrease the total number of
connected components of $\cup_{i=1}^r Q^{-1}(x_i)$.

In this way, we get a continuous map $Q$
homotopic to the original one such that
$Q$ is transverse to $x_1, x_2, \ldots, x_r$ and that
$Q^{-1}(x_i)$ are all connected.

Let $M'$ be the compact $m$--dimensional manifold
with boundary obtained by cutting $M$ along
$\cup_{i=1}^r Q^{-1}(x_i)$. Note that $M'$
is connected. Let $G'$ be the graph
obtained by cutting $G$ at $x_1, x_2, \ldots, x_r$.
Note that $G'$ is a tree and has $2r$ \emph{distinguished vertices}
corresponding to $x_1, x_2, \ldots, x_r$. We denote
the distinguished vertices corresponding to $x_i$ by
$x_{i+}$ and $x_{i-}$, $i = 1, 2, \ldots, r$.
Then, we have a natural
continuous map $Q' : M' \to G'$ induced by $Q$, where
$\partial M' = \cup_{i=1}^r ((Q')^{-1}(x_{i+}) \cup
(Q')^{-1}(x_{i-}))$.
In the following, we will construct disjoint $(m-1)$-dimensional
closed connected submanifolds of $M'$ that correspond bijectively
to the edges of $G'$ in such a way that the closures of the connected
components of the complement in $M'$ correspond bijectively to
the vertices of $G'$ and that a condition similar to (\ref{cob})
is satisfied for each vertex of $G'$ except for the distinguished
vertices.

First, for each edge incident to a distinguished
vertex $x_{i\pm}$, we associate to the edge
a submanifold in $\Int{M'}$ parallel and close
to the boundary component $(Q')^{-1}(x_{i\pm})$ of $M'$.
Note that
if an edge is incident to two distinct distinguished
vertices, then $G$ must be a circle and has no vertex, which
is a contradiction. So, this submanifold is well defined.

Let $e$ be an edge, not incident to a distinguished vertex.
We denote by $x_e$ a point in the interior of $e$.
Since $G'$ is a tree, $G' \setminus \{x_e\}$
has exactly two connected components, say $G'_1$ and $G'_2$.
Let $V_j$ denote the set of
distinguished vertices belonging to $G'_j$, $j = 1, 2$.
Then, we can construct an $(m-1)$--dimensional closed
connected submanifold $E_e$ of $M'$ such that
\begin{enumerate}
\item $E_e \subset \Int{M'}$,
\item $E_e$ is disjoint from the submanifolds
corresponding to the edges incident to distinguished vertices, 
\item $M' \setminus E_e$ has exactly two
connected components, say $M'_1$ and $M'_2$, and
\item the submanifolds corresponding to 
the edges incident to distinguished vertices in $V_j$ 
are contained in $M'_j$,
$j = 1, 2$, after renumbering $M'_1$ and $M'_2$ if necessary.
\end{enumerate}

Such a submanifold $E_e$ can be constructed, for example,
as follows. Let $F_k$, $k = 1, 2, \ldots, a$,
be the $(m-1)$--dimensional submanifolds associated with
the edges incident to the distinguished vertices in $V_1$.
Since $M'$ is connected, and $F_k$ are parallel to
boundary components, we can find smoothly embedded arcs
$\alpha_k$, $k = 1, 2, \ldots, a-1$, in $M'$ such that
\begin{enumerate}
\item $\alpha_k$ intersects $F_k$ and $F_{k+1}$
exactly at the end points and the intersections are transverse,
\item $\alpha_k$ does not intersect $F_j$, $j \neq k, k+1$,
\item $\alpha_1, \alpha_2, \ldots, \alpha_{a-1}$
are disjoint.
\end{enumerate}
If $m = \dim{M'} \geq 3$, then such arcs as above can easily be found.
When $m=2$, consider an arbitrary properly embedded 
arc in a compact connected surface such that the end points
lie in different components of the boundary. Then, such an arc
is never separating. Therefore, a set of arcs as above
can be found for this case as well.
Now, consider small tubular neighborhoods $h_k
\cong [0, 1] \times D^{m-1}$ of $\alpha_k$
as $1$--handles attached to $F_1, F_2, \ldots, F_a$,
and use them to perform surgery on $F_0 = \cup_{k=1}^a F_k$ 
so that we get
$$E_e = (F_0 \setminus ((\cup_{k=1}^{a-1}h_k) \cap F_0))
\cup (\cup_{k=1}^{a-1} dh_k),$$
where $dh_k \cong [0, 1] \times \partial D^{m-1}$.
After the surgery, we smooth the corners so that $E_e$ 
is a smoothly embedded $(m-1)$--dimensional submanifold of $M'$.
By construction, it is connected and closed.
We can further move $E_e$ by isotopy so that it is
disjoint from $\cup_{k=1}^a F_k$.
Then, we can easily check that $E_e$ has the
desired properties.

For the moment, we ignore $Q'$, and cut $M'$ along
$E_e$ to get two compact connected manifolds.
We also cut $G'$ at $x_e$ into two trees, where
each tree has an additional distinguished vertex.
Then, we continue
the same procedures to get an $(m-1)$--dimensional
connected closed submanifold corresponding to an edge
not incident to a distinguished vertex. As the number of edges is finite,
this process will terminate in a finite number of steps.
Finally, we get a family of closed connected $(m-1)$--dimensional
submanifolds of $M'$
that correspond bijectively to the edges of $G'$.
By construction, each component of the complement 
corresponds to a unique vertex of $G'$ in such a way that
the closure contains an $(m-1)$--dimensional
submanifold $E_e$ corresponding to an edge $e$
if and only
if the vertex is incident to the edge $e$.

Recall that $M$ can be reconstructed
from $M'$ by identifying pairs of boundary components.
In this sense, we identify the submanifolds in $\Int{M'}$
with those in $M$.
Likewise, $G$ is also reconstructed from $G'$.
Now we associate the above constructed $(m-1)$--dimensional
submanifold $E_e$ to each edge $e$ of $G$, where
for the edge $e$ containing $x_i$,
we put $E_e = Q^{-1}(x_i)$, $i = 1, 2, \ldots, r$.
Furthermore, we associate to each vertex of $G$
the closure of the component as described in the
previous paragraph. 
All these ingredients show that we have
constructed an $m$--decorated graph as in \S\ref{section4}.
By construction, the condition (\ref{cob}) is automatically
satisfied.

Then, by using the techniques used in the proof of Theorem~\ref{thm2},
we can construct a smooth function $f : M'' \to \R$
with finitely many critical values that realizes the
$m$--decorated graph as the Reeb space. By using the original identification
maps for diffeomorphisms for gluing the pieces when
constructing $M''$, we can arrange so that $M''$ is naturally identified 
with $M$.
Then, the Reeb space $W_f$ can also be naturally identified with $G$.

Finally, we should note that the quotient map
$q_f : M \to W_f = G$ is homotopic to $Q$.
This follows from the fact that $G \setminus \{x_1, x_2, \ldots, x_r\}$
is contractible.
This completes the proof.
\end{proof}

For a finitely generated group $H$, set
$$\mathrm{corank}(H) = \max\{r \,|\, \text{There exists
an epimorphism } H \to F_r\},$$
where $F_r$ is the free group of rank $r \geq 0$.
This is called the \emph{co-rank} of the
group $H$ (for example, see \cite{G1, G2}).

As an immediate corollary to Theorem~\ref{thm3},
we get the following.

\begin{cor}
Let $M$ be a smooth closed connected manifold
of dimension $m \geq 2$, and
$G$ a connected graph
without loops.
Then, $G$ arises as the Reeb space
of a certain smooth function on $M$ with
finitely many critical values if and only if
$\beta_1(G) \leq \mathrm{corank}(\pi_1(M))$, where
$\beta_1$ denotes the first betti number.
\end{cor}

\begin{rem}
About realization of Reeb graphs, 
there have been a lot of studies, e.g.\ by
Sharko \cite{Sh}, Mart\'\i nez-Alfaro, Meza-Sarmiento and Oliveira
\cite{MMO1, MMO2, MMO3}, Masumoto and Saeki \cite{MS}, Gelbukh \cite{G1,
G2, G3, G4}, Kaluba, Marzantowicz and Silva \cite{KMS},
Michalak \cite{M1, M2}, Michalak
and Marzantowicz \cite{MM}, 
Kitazawa \cite{K1, K2, K3}, etc.
Our theorems generalize some of them.
\end{rem}

\section*{Acknowledgment}\label{ack}
The author would like to thank Dr.\ Naoki Kitazawa and
Dr.\ Dominik Wrazidlo
for helpful discussions.
This work was supported by JSPS KAKENHI Grant Number 
JP17H06128. This work was also supported by the Research  
Institute for Mathematical Sciences, an International Joint
Usage/Research Center located in Kyoto University. 



\begin{thebibliography}{99999}









\bibitem{CGM}J.S.~Calcut, R.E.~Gompf and J.D.~McCarthy, 
{\em On fundamental groups of quotient spaces}, Topol.\ Appl.\ 
\textbf{159} (2012), 322--330.












\bibitem{G1}I.~Gelbukh,
{\em Co-rank and Betti number of a group},
Czechoslovak Math.\ J.\ \textbf{65 (140)} (2015), 565--567.

\bibitem{G2}I.~Gelbukh,
{\em The co-rank of the fundamental group: the direct product, the first Betti number, and the topology of foliations},
Math.\ Slovaca \textbf{67} (2017), 645--656.

\bibitem{G3}I.~Gelbukh, {\em
Loops in Reeb graphs of $n$--manifolds}, 
Discrete Comput.\ Geom.\ \textbf{59} (2018), 843--863.

\bibitem{G4}I.~Gelbukh,
{\em Approximation of metric spaces by Reeb graphs: Cycle rank of a 
Reeb graph, the 
co-rank of the fundamental group, and large components of level 
sets on Riemannian manifolds},
Filomat \textbf{33} (2019), 2031--2049. 

\bibitem{HS}J.T.~Hiratuka and O.~Saeki, 
{\em Triangulating Stein factorizations of generic maps and 
Euler characteristic formulas},
in ``Singularity theory, geometry and topology'', pp.~61--89,
RIMS K\^{o}ky\^{u}roku Bessatsu \textbf{B38}, 2013. 

\bibitem{Izar1}S.A.~Izar, {\em Fun\c{c}\~oes de Morse: um teorema
de classifica\c{c}\~oes em dimens\~ao $2$}, Master Thesis,
University of S\~ao Paulo, 1978.
%
\bibitem{Izar2}S.A.~Izar, {\em Fun\c{c}\~oes de Morse: uma teoria
combinat\'oria em dimens\~ao tr\^es}, Doctor Thesis, University of
S\~ao Paulo, 1983.
%
\bibitem{Izar3}S.A.~Izar, {\em Fun\c{c}\~oes de Morse e topologia das
superf\'icies I: O grafo de Reeb de $f : M \to \R$}, 
M\'etrica no.~31, Estudo e Pesquisas em Matem\'atica, IBILCE,
Brazil, 1988. \\
\verb+ https://www.ibilce.unesp.br/Home/Departamentos/Matematica/metrica-31.pdf+
%
\bibitem{Izar4}S.A.~Izar, {\em Fun\c{c}\~oes de Morse e topologia das
superf\'icies II: Classifica\c{c}\~ao das fun\c{c}\~oes de Morse
est\'aveis sobre superf\'icies}, 
M\'etrica no.~35, Estudo e Pesquisas em Matem\'atica, IBILCE,
Brazil, 1989. \\
\verb+ https://www.ibilce.unesp.br/Home/Departamentos/Matematica/metrica-35.pdf+
%
\bibitem{Izar5}S.A.~Izar, {\em Fun\c{c}\~oes de Morse e topologia das
superf\'icies III: Campos pseudo-gradientes de uma fun\c{c}\~ao
de Morse sobre uma superf\'icie}, M\'etrica no.~44, 
Estudo e Pesquisas em Matem\'atica, IBILCE,
Brazil, 1992. \\
\verb+ https://www.ibilce.unesp.br/Home/Departamentos/Matematica/metrica-44.pdf+

\bibitem{KMS}M.~Kaluba, W.~Marzantowicz and N.~Silva, 
{\em On representation of the Reeb graph as a sub-complex of manifold}, 
Topol.\ Methods Nonlinear Anal.\ \textbf{45} (2015), 287--307.

\bibitem{K1}N.~Kitazawa,
{\em On Reeb graphs induced from smooth functions on $3$--dimensional 
closed orientable manifolds with finitely many 
singular values}, preprint, arXiv:1902.08841 [math.GT].

\bibitem{K2}N.~Kitazawa,
{\em On Reeb graphs induced from smooth functions on closed or open manifolds},
preprint, arXiv:1908.04340 [math.GT].

\bibitem{K3}N.~Kitazawa,
{\em Maps on manifolds onto graphs locally regarded as a quotient 
map onto a Reeb space and construction problem},
preprint, arXiv:1909.10315 [math.GT].

\bibitem{L1}H.~Levine, {\sl Classifying immersions into $\R^4$
over stable maps of $3$-manifolds into $\R^2$},
Lecture Notes in Math., Vol.~1157, Springer--Verlag, Berlin, 1985.

\bibitem{MM}W.~Marzantowicz and \L.P.~Michalak,
{\em Relations between Reeb graphs, systems of hypersurfaces and 
epimorphisms onto free groups},
preprint, arXiv:2002.02388 [math.GT].

\bibitem{MMO1}J.~Mart\'{\i}nez-Alfaro, I.S.~Meza-Sarmiento and R.D.S.~Oliveira, 
{\em Singular levels and topological invariants of Morse Bott 
integrable systems on surfaces},
J.\ Differential Equations \textbf{260} (2016), 688--707. 

\bibitem{MMO2}J.~Mart\'{\i}nez-Alfaro, I.S.~Meza-Sarmiento and R.D.S.~Oliveira, 
{\em Topological classification of simple Morse Bott functions on surfaces},
in ``Real and complex singularities'', pp.~165--179,
Contemp.\ Math., Vol.~675, Amer.\ Math.\ Soc., Providence, RI, 2016. 

\bibitem{MMO3}J.~Mart\'{\i}nez-Alfaro, I.S.~Meza-Sarmiento and R.D.S.~Oliveira, 
{\em Singular levels and topological invariants of Morse-Bott foliations 
on non-orientable surfaces},
Topol.\ Methods Nonlinear Anal.\ \textbf{51} (2018), 183--213. 

\bibitem{MS}Y.~Masumoto and O.~Saeki, 
{\em A smooth function on a manifold with given Reeb graph},
Kyushu J.\ Math.\ \textbf{65} (2011), 75--84. 

\bibitem{M1}\L.P.~Michalak, 
{\em Realization of a graph as the Reeb graph of a Morse function on a manifold},
Topol.\ Methods Nonlinear Anal.\ \textbf{52} (2018), 749--762. 

\bibitem{M2}\L.P.~Michalak, 
{\em Combinatorial modifications of Reeb graphs and the realization problem},
preprint, arXiv:1811.08031 [math.GT].

\bibitem{Milnor1}J.W.~Milnor, 
{\em Sommes de vari\'{e}tes diff\'{e}rentiables et structures 
diff\'{e}rentiables des sph\`{e}res}, 
Bull.\ Soc.\ Math.\ France \textbf{87} (1959), 439--444. 

\bibitem{Milnor2}J.~Milnor, {\em Differential topology}, 
Lectures on Modern Mathematics, Vol.~II, 1964, pp.~165--183, 
Wiley, New York.

\bibitem{R}G.~Reeb, 
{\em Sur les points singuliers d'une forme de 
Pfaff compl\`{e}tement int\'{e}grable ou d'une fonction 
num\'{e}rique}, Comptes Rendus Hebdomadaires des S\'{e}ances 
de l'Acad\'{e}mie des Sciences \textbf{222} (1946), 847--849. 




\bibitem{Sa20}O.~Saeki, {\em Reeb graphs of smooth functions on
manifolds}, preprint, May 2020, to appear in RIMS K\^{o}ky\^{u}roku.

\bibitem{Sharko}V.V.~Sharko, {\em About Kronrod-Reeb 
graph of a function on a manifold},
Methods of Functional Analysis and Topology 
\textbf{12} (2006), 389--396.

\bibitem{Sh}M.~Shiota, {\em Thom's conjecture on triangulations 
of maps}, Topology \textbf{39} (2000), 383--399.



\bibitem{Wh1934}H.~Whitney, 
{\em Analytic extensions of differentiable functions defined in closed sets},
Trans.\ Amer.\ Math.\ Soc.\ \textbf{36} (1934), 63--89. 




\end{thebibliography}
\end{document}